\documentclass[11pt]{amsart}

\usepackage{amsmath,amsfonts,amssymb,amsthm,amscd,comment,textcomp,mdwlist}
\usepackage[colorlinks=true, pdfstartview=FitV, linkcolor=blue, citecolor=blue, urlcolor=blue, breaklinks=true]{hyperref}
\usepackage[all]{xy}
\usepackage{etoolbox}

\theoremstyle{plain}
\newtheorem{thm}{Theorem}[section]
\newtheorem{prop}[thm]{Proposition}
\newtheorem{lemma}[thm]{Lemma}
\newtheorem{cor}[thm]{Corollary}

\theoremstyle{definition}
\newtheorem{defn}[thm]{Definition}
\newtheorem{rmk}[thm]{Remark}
\newtheorem{ex}[thm]{Example}

\numberwithin{equation}{section}

\newcommand{\g}{\mathfrak{g}}
\newcommand{\h}{\mathfrak{h}}

\newcommand{\n}{\mathfrak{n}}
\newcommand{\q}{\mathfrak{q}}
\newcommand{\gl}{\mathfrak{gl}}

\newcommand{\Z}{\mathbb{Z}}
\newcommand{\C}{\mathbb{C}}
\newcommand{\N}{\mathbb{N}}

\newcommand{\sm}{\mathsf{m}}
\newcommand{\sM}{\mathsf{M}}

\DeclareMathOperator{\Ann}{Ann}
\DeclareMathOperator{\Hom}{Hom}

\DeclareMathOperator{\End}{End}

\DeclareMathOperator{\Spec}{Spec}
\DeclareMathOperator{\MaxSpec}{MaxSpec}

\DeclareMathOperator{\ev}{ev}
\DeclareMathOperator{\Supp}{Supp}
\DeclareMathOperator{\id}{id}
\DeclareMathOperator{\hgt}{ht}

\newtoggle{comments}
\newtoggle{details}
\newtoggle{prelimnote}
\newtoggle{detailsnote}

\toggletrue{comments}   

\toggletrue{details}   

\toggletrue{prelimnote}   
\togglefalse{prelimnote} 

\toggletrue{detailsnote}   
\togglefalse{detailsnote} 

\iftoggle{comments}{%
  \newcommand{\comments}[1]{
    \begin{center}
      \parbox{6.5 in}{
        \color{red}
          {\footnotesize \textbf{Comments:} #1}
        \color{black}}
    \end{center}}
}{%
  \newcommand{\comments}[1]{}
}

\iftoggle{details}{%
  \newcommand{\details}[1]{
      \ \\
      \color{OliveGreen}
        {\footnotesize \textbf{Details:} #1}
      \color{black}
      \\
  }
}{%
  \newcommand{\details}[1]{}
}

\iftoggle{prelimnote}{%
  
}{%
  
}

\begin{document}
%

\title{Equivariant map queer Lie superalgebras}

\author{Lucas Calixto}
\address{L.~Calixto: UNICAMP - IMECC, Campinas - SP - Brazil, 13083-859}
\email{lhcalixto@ime.unicamp.br}
\thanks{The first author was supported by FAPESP grant 2013/08430-4.}

\author{Adriano Moura}
\address{A.~Moura: UNICAMP - IMECC, Campinas - SP - Brazil, 13083-859}
\urladdr{\url{http://www.ime.unicamp.br/~aamoura/}}
\email{aamoura@ime.unicamp.br}
\thanks{The second author was partially supported by CNPq grant 303667-2011/7 and FAPESP grant 2014/09310-5.}

\author{Alistair Savage}
\address{A.~Savage: Department of Mathematics and Statistics, University of Ottawa, Canada}
\urladdr{\url{http://alistairsavage.ca}}
\email{alistair.savage@uottawa.ca}
\thanks{The third author was supported by a Discovery Grant from the Natural Sciences and Engineering Research Council of Canada.}

\begin{abstract}
  An equivariant map queer Lie superalgebra is the Lie superalgebra of regular maps from an algebraic variety (or scheme) $X$ to a queer Lie superalgebra $\q$ that are equivariant with respect to the action of a finite group $\Gamma$ acting on $X$ and $\q$.  In this paper, we classify all irreducible finite-dimensional representations of the equivariant map queer Lie superalgebras under the assumption that $\Gamma$ is abelian and acts freely on $X$.  We show that such representations are parameterized by a certain set of $\Gamma$-equivariant finitely supported maps from $X$ to the set of isomorphism classes of irreducible finite-dimensional representations of $\q$.  In the special case where $X$ is the torus, we obtain a classification of the irreducible finite-dimensional representations of the twisted loop queer superalgebra.
\end{abstract}



\maketitle
\thispagestyle{empty}

\setcounter{tocdepth}{1}

\tableofcontents

%
\section{Introduction}
%

Equivariant map algebras can be viewed as a generalization of (twisted) current algebras and loop algebras.  Namely, let $X$ be an algebraic variety (or, more generally, a scheme) and let $\g$ be a finite-dimensional Lie algebra, both defined over the field of complex numbers. Furthermore, suppose that a finite group $\Gamma$ acts on both $X$ and $\g$ by automorphisms. Then the equivariant map algebra $M(X,\g)^{\Gamma}$ is defined to be the Lie algebra of $\Gamma$-equivariant regular maps from $X$ to $\g$.  Equivalently, consider the induced action of $\Gamma$ on the coordinate ring $A$ of $X$. Then $M(X,\g)^{\Gamma}$ is isomorphic to $(\g\otimes A)^{\Gamma}$, the Lie algebra of fixed points of the diagonal action of $\Gamma$ on $\g\otimes A$.  Recently, the representation theory of equivariant map algebras, either in full generality or in special cases, has been the subject of much research.  We refer the reader to the survey~\cite{erh-sav13} for an overview.

Lie superalgebras are generalizations of Lie algebras and are an important tool for physicists in the study of supersymmetries.  The finite-dimensional simple complex Lie superalgebras were classified by Victor Kac in \cite{kac77}, and the irreducible finite-dimensional representations of the so-called basic classical Lie superalgebras were classified in \cite{kac77} and \cite{Kac78}.  It is thus natural to consider equivariant map superalgebras, where the target Lie algebra $\g$ mentioned above is replaced by a finite-dimensional Lie superalgebra.  In \cite{sav14}, the third author classified the irreducible finite-dimensional representations of $M(X,\g)^\Gamma$ when $\g$ is a basic classical Lie superalgebra, $X$ has a finitely-generated coordinate ring, and $\Gamma$ is an abelian group acting freely on the set of rational points of $X$. These assumptions make much of the theory parallel to the non-super setting.  The goal of the current paper is to move beyond the setting of basic classical Lie superalgebras.  In particular, we address the case where $\g$ is the so-called \emph{queer Lie superalgebra}.  In this case, very little is known about the representation theory of the equivariant map Lie superalgebra, even when $\Gamma$ is trivial or $X$ is the affine plane or torus (the current and loop cases, respectively), although the representations of the corresponding affine Lie superalgebra have been studied in \cite{GS08}.

The queer Lie superalgebra $\q(n)$ was introduced by Victor Kac in \cite{kac77}.  It is a simple subquotient of the Lie superalegbra of endomorphisms of $\C^{n|n}$ that commute with an odd involution (see Remark~\ref{rem:queer-def}).  It is closely related to the Lie algebra $\frak{sl}(n+1)$, in the sense that $\q(n)$ is a direct sum of one even and one odd copy of $\frak{sl}(n+1)$.  Although the queer Lie superalgebra is classical, its properties are quite different from those of the other classical Lie superalgebras.  In particular, the Cartan subalgebra of $\q(n)$ is not abelian.  (Here, and throughout the paper, we use the term \emph{subalgebra} even in the super setting, and avoid the use of the cumbersome term \emph{subsuperalgebra}.)  For this reason, the corresponding theory of weight modules is much more complicated.  The theory requires Clifford algebra methods, since the highest weight space of an irreducible highest weight $\q(n)$-module has a Clifford module structure.  Nevertheless, the theory of finite-dimensional $\q(n)$-modules is well developed (see, for example, \cite{pen86,pen-ser97,gor06}).  It is the fact that the queer Lie superalgebra is similar to the Lie algebra $\mathfrak{sl}(n+1)$ in some ways while, on the other hand, having very different structure and representation theory that explains the special attention this Lie superalgebra has received.

To investigate the representation theory of the Lie superalgebra $\q(n) \otimes A$, where $A$ is a commutative unital associative algebra, the first step is understanding the irreducible finite-dimensional representations of its Cartan subalgebra $\h \otimes A$, where $\h$ is the standard Cartan subalgebra of $\q(n)$.  Therefore, in the current paper, we first give a characterization of the irreducible finite-dimensional $\h\otimes A$-modules (Theorem~\ref{class.cartan.mod}).  Next, we give a characterization of \emph{quasifinite} irreducible highest weight $\q(n)\otimes A$-modules in Theorem~\ref{thm:quasifinite}.  Such modules generalize finite-dimensional modules.  Using these results, we are able to give a complete classification of the irreducible finite-dimensional representations of the equivariant map queer Lie superalgebra in the case that the algebra $A$ is finitely generated and the group $\Gamma$ is abelian and acts freely on $\MaxSpec(A)$.  Our main result, Theorem~\ref{thm:twisted-classification}, is that the irreducible finite-dimensional modules are parameterized by a certain set of $\Gamma$-equivariant finitely supported maps defined on $\MaxSpec(A)$.  In the special cases that $X$ is the torus or affine line, our results yield a classification of the irreducible finite-dimensional representations of the twisted loop queer Lie superalgebra and twisted current queer Lie superalgebra, respectively.

This paper is organized as follows. In Section~\ref{sec:prelim}, we review some results on commutative algebras, associative superalgebras (in particular Clifford algebras) and Lie superalgebras (especially the queer Lie superalgebra).  We introduce the equivariant map Lie superalgebras in Section~\ref{Sec:EMQLS}.  In Section~\ref{sec:cartan-modules}, we give a characterization of the irreducible finite-dimensional representations of the Lie superalgebra $\h\otimes A$.  In Section~\ref{sec:hw-modules}, we define quasifinite and highest weight modules and we give a characterization of the quasifinite modules. In Section~\ref{sec:evaluation}, we introduce evaluation representations and their irreducible products.  These play a key role in our classification. Finally, in Section~\ref{sec:fd-rep-classification}, we classify all the irreducible finite-dimensional modules of $\q(n) \otimes A$ and $(\q(n) \otimes A)^\Gamma$.

\subsection*{Notation}

We let $\Z$ be the ring of integers, $\N$ be the set of nonnegative integers and $\Z_2=\{{\bar 0},{\bar 1}\}$ be the quotient ring $\Z/2\Z$. Vector spaces, algebras, tensor products, etc.\ are defined over the field of complex numbers $\C$ unless otherwise stated.  Whenever we refer to the dimension of an algebra or ideal, we refer to its dimension over $\C$.

\subsection*{Acknowledgements}

The third author would like to thank S.-J.\ Cheng for helpful conversations.

%
\section{Preliminaries}\label{sec:prelim}
%

\subsection{Associative superalgebras} \label{subsec:assoc-alg}

We collect here some results that will be used in the sequel.  Let $A$ denote a commutative associative unital algebra and let $\MaxSpec(A)$ be the set of all maximal ideals of $A$.

\begin{defn}[$\Supp(I)$] \label{def:ideal-support}
  The \emph{support} of an ideal $I\subseteq A$ is defined to be the set
  \[
    \Supp(I)=\{\sm \in {\MaxSpec(A)\,|\,I\subseteq \sm \}}.
  \]
\end{defn}

A proof of the following lemma can be found, for instance, in \cite[\S 2.1]{sav14}.

\begin{lemma} \label{lem:assoc-alg-facts} Let $I$ and $J$ be ideals of $A$. Then,
  \begin{enumerate}
    \item \label{lem-item:finite-codim-finite} If $I$ is of finite codimension, then $\Supp(I)$ is finite.

    \item \label{lem-item:finite-finite-codim} If $A$ is finitely generated and $I$ has finite support, then $I$ is of finite codimension in $A$.

    \item \label{lem-item:product-intersection} If $I$ and $J$ have disjoint supports, then $IJ = I \cap J$.

    \item \label{lem-item:Noetherian-power-radical} If $A$ is Noetherian, then every ideal of $A$ contains a power of its radical.
  \end{enumerate}
\end{lemma}

Now let $V = V_{\bar 0}\oplus V_{\bar 1}$ be a $\Z_2$-graded vector space. The parity of a homogeneous element $v\in V_i$ will be denoted by $|v|=i$, $i\in \Z_2$. An element in $V_{\bar 0}$ is called \emph{even}, while an element in $V_{\bar 1}$ is called \emph{odd}.  A \emph{subspace} of $V$ is a $\Z_2$-graded vector space $W=W_{\bar 0}\oplus W_{\bar 1}\subseteq V$ such that $W_i\subseteq V_i$ for $i\in \Z_2$.  We denote by $\C^{m|n}$ the vector space $\C^m \oplus \C^n$, where the first summand is even and the second is odd.

An \emph{associative superalgebra} $A$ is a $\Z_2$-graded vector space $A=A_{\bar 0}\oplus A_{\bar 1}$ equipped with a bilinear associative multiplication (with unit element) such that $A_iA_j\subseteq A_{i+j}$, for $i,j\in \Z_2$. A homomorphism between two superalgebras $A$ and $B$ is a map $g \colon A \rightarrow B$ which is a homomorphism between the underlying algebras, and, in addition, $g(A_i)\subseteq B_i$ for $i\in \Z_2$.  The tensor product $A\otimes B$ is the superalgebra whose vector space is the tensor product of the vector spaces of $A$ and $B$, with the induced $\Z_2$-grading and multiplication defined by $(a_1\otimes b_1)(a_2\otimes b_2)=(-1)^{|a_2||b_1|}a_1a_2\otimes b_1b_2$, for homogeneous elements $a_i\in A$, and $b_i\in B$. An $A$-module $M$ is always understood in the $\Z_2$-graded sense, that is, $M=M_{\bar 0}\oplus M_{\bar 1}$ such that $A_iM_j\subseteq M_{i+j}$, for $i,j\in\Z_2$. Subalgebras and ideals of superalgebras are $\Z_2$-graded subalgebras and ideals. A superalgebra that has no nontrivial ideal is called \emph{simple}.  A homomorphism between $A$-modules $M$ and $N$ is a linear map $f \colon M \rightarrow N$ such that $f(xm)=xf(m)$, for all $x\in A$ and $m\in M$. A homomorphism is of degree $|f|\in \Z_2$, if $f(M_i)\subseteq N_{i+|f|}$ for $i\in \Z_2$.

We denote by $M(m|n)$ the superalgebra of complex matrices in $m|n$-block form
\[
  \left(\begin{array}{r|r}
    A & B  \\
    \hline
    C & D \\
  \end{array}\right),
\]
whose even subspace consists of the matrices with $B=0$ and $C=0$, and whose odd subspace consists of the matrices with $A=0$ and $D=0$.  If $V=V_{\bar 0}\oplus V_{\bar 1}$ is a $\Z_2$-graded vector space with $\dim V_{\bar 0}=m$ and $\dim V_{\bar 1}=n$, then the endomorphism superalgebra $\End(V)$ is the associative superalgebra of endomorphisms of $V$, where
\[
  \End(V)_{i}=\{ T\in \End(V)\ |\ T(V_j)\subseteq V_{i+j},\ j \in \Z_2 \},\quad i \in \Z_2.
\]
Note that fixing ordered bases for $V_{\bar 0}$ and $V_{\bar 1}$ gives an isomorphism between $\End(V)$ and $M(m|n)$.

For $m \ge 1$, let $P\in M(m|m)$ be the matrix
\[
  \left(\begin{array}{r|r}
    0 & I_m  \\
    \hline
    -I_m & 0 \\
  \end{array}\right),
\]
and define $Q(m)_i := \{T\in M(m|m)_i\ |\ TP=(-1)^i PT\}$, for $i\in \Z_2$. Then $Q(m)=Q(m)_{\bar 0}\oplus Q(m)_{\bar 1}$ is the subalgebra of $M(m|m)$ consisting of matrices of the form
\begin{equation}\label{matrixqueer}
  \left(\begin{array}{r|r}
    A & B  \\
    \hline
    B & A \\
  \end{array}\right),
\end{equation}
where $A$ and $B$ are arbitrary $m\times m$ matrices.

\begin{thm}[{\cite[p.~94]{chewan12}}] \label{thm:simplemod}
  Consider $\C^{m|n}$ as an $M(m|n)$-module via matrix multiplication. Then the unique irreducible finite-dimensional module, up to isomorphism, of $M(m|n)$ $($resp. $Q(m))$ is $\C^{m|n}$ $($resp. $\C^{m|m})$.
\end{thm}

For an associative superalgebra $A$, we shall denote by $|A|$ the underlying (i.e.\ ungraded) algebra. Denote by $Z(|A|)$ the center of $|A|$.  Note that $Z(|A|)=Z(|A|)_{\bar 0}\oplus Z(|A|)_{\bar 1}$, where $Z(|A|)_{i}=Z(|A|)\cap A_i$, for $i\in \Z_2$.

\begin{thm}[{\cite[Th.~3.1]{chewan12}}] \label{thm:simplealg}
  Let $A$ be a finite-dimensional simple associative superalgebra.
  \begin{enumerate}
    \item If $Z(|A|)_{\bar 1}=0$, then $A$ is isomorphic to $M(m|n)$, for some $m$ and $n$.
    \item If $Z(|A|)_{\bar 1}\ne 0$, then $A$ is isomorphic to $Q(m)$, for some $m$.
  \end{enumerate}
\end{thm}

\begin{defn}[Clifford algebra]
  Let $V$ be a finite-dimensional vector space and $f \colon V \times V \rightarrow \C$ be a symmetric bilinear form. We call the pair $(V,f)$ a \emph{quadratic pair}. Let $J$ be the ideal of the tensor algebra $T(V)$ generated by the elements
  \[
    x \otimes x-f(x,x)1,\quad x \in V,
  \]
  and set $C(V,f) := T(V)/J$.  The algebra $C(V,f)$ is called the \emph{Clifford algebra} of the pair $(V,f)$ over $\C$.
\end{defn}

\begin{rmk}[{\cite[Ch.~12, Def.~4.1 and Th.~4.2]{hus94}}]\label{univ-cliff-prop}
  For a quadratic pair $(V,f)$, there exists a linear map $\theta \colon V \rightarrow C(V,f)$ such that the pair $(C(V,f),\theta)$ has the following universal property: For all linear maps $u \colon V \to A$ such that $u(v)^2 = f(v,v) 1_A$ for all $v \in V$, where $A$ is a unital algebra, there exists a unique algebra homomorphism $u' \colon C(V,f) \to A$ such that $u'\theta = u$.
\end{rmk}

Clifford algebras also have a natural superalgebra structure. Indeed, $T (V)$ possesses a $\Z_2$-grading (by even and odd tensor powers) such that $J$ is homogeneous, so the grading descends to $C(V,f)$.  Thus, the resulting superalgebra $C(V,f)$ is sometimes called the \emph{Clifford superalgebra}.

\begin{lemma}[{\cite[Th.~A.3.6]{mus12}}] \label{lem:simpleclif}
  Let $(V,f)$ be a quadratic pair with $f$ nondegenerate. Then $C(V,f)$ is a simple associative superalgebra.
\end{lemma}

\begin{rmk}\label{rmksimpleclif}
  It follows from Lemma~\ref{lem:simpleclif}, together with Theorems~\ref{thm:simplemod} and~\ref{thm:simplealg}, that any Clifford superalgebra associated to a nondegenerate pair (i.e.\ the symmetric bilinear form associated to this pair is nondegenerate) has only one irreducible finite-dimensional module up to isomorphism.
\end{rmk}

\subsection{Lie superalgebras}

\begin{defn}[Lie superalgebra]
  A \emph{Lie superalgebra} is a $\Z_2$-graded vector space $\g=\g_{\bar 0}\oplus \g_{\bar 1}$ with bilinear multiplication $[\cdot,\cdot]$ satisfying the following axioms:
  \begin{enumerate}
    \item The multiplication respects the grading: $[\g_i,\g_j] \subseteq \g_{i+j}$ for all $i,j \in \Z_2$.
    \item Skew-supersymmetry: $[a,b]=-(-1)^{|a||b|}[b,a]$, for all homogeneous elements $a,b\in \g$.
    \item Super Jacobi Identity: $[a,[b,c]]=[[a,b],c]+(-1)^{|a||b|}[b,[a,c]]$, for all homogeneous elements $a,b,c\in \g$.
  \end{enumerate}
\end{defn}

\begin{ex}
  Let $A=A_{\bar 0}\oplus A_{\bar 1}$ be an associative superalgebra. We can make $A$ into a Lie superalgebra by letting $[a,b]:=ab-(-1)^{|a||b|}ba$, for all homogeneous $a,b\in A$, and extending $[\cdot,\cdot]$ by linearity. We call this the \emph{Lie superalgebra associated to} $A$. The Lie superalgebra associated to $\End(V)$ (resp. $M(m|n)$) is called the \emph{general linear Lie superalgebra} and is denoted by $\gl(V)$ (resp $\gl(m|n)$).
\end{ex}

Just as for Lie algebras, a finite-dimensional Lie superalgebra $\g$ is said to be \emph{solvable} if $\g^{(n)}=0$ for some $n \geq 0$, where we define inductively $\g^{(0)}=\g$ and $\g^{(n)}=[\g^{(n-1)},\g^{(n-1)}]$ for $n \ge 1$.

\begin{lemma}[{\cite[Prop.~5.2.4]{kac77}}]\label{lem:solvablemod}
  Let $\g=\g_{\bar{0}}\oplus \g_{\bar{1}}$ be a finite-dimensional solvable Lie superalgebra such that $[\g_{\bar{1}},\g_{\bar{1}}]\subseteq [\g_{\bar{0}},\g_{\bar{0}}]$. Then every irreducible finite-dimensional $\g$-module is one-dimensional.
\end{lemma}

\begin{lemma}[{\cite[Lem.~2.6]{sav14}}] \label{lem:ideal-annihilate-vector}
  Suppose $\g$ is a Lie superalgebra and $V$ is an irreducible $\g$-module such that $\mathfrak{J}v=0$ for some ideal $\mathfrak{J}$ of $\g$ and nonzero vector $v\in V$. Then $\mathfrak{J}V=0$.
\end{lemma}

The next two results are super versions of well-known results in representation theory. Namely, the Poincaré-Birkhoff-Witt Theorem (or PBW Theorem) and Schur's Lemma, respectively.

\begin{lemma}[{\cite[Th.~6.1.1]{mus12}}] \label{lem:PBW}
  Let $\g=\g_{\bar{0}}\oplus \g_{\bar{1}}$ be a Lie superalgebra and let $B_0$, $B_1$ be totally ordered bases for $\g_{\bar{0}}$ and $\g_{\bar{1}}$, respectively. Then the monomials
  \[
    u_1\cdots u_rv_1\cdots v_s,\quad u_i\in B_0,\quad v_i\in B_1\quad\text{and}\quad u_1\leq \cdots \leq u_r, \,\, v_1<\cdots < v_s,
  \]
  form a basis of the universal enveloping superalgebra $U(\g)$.
\end{lemma}

\begin{lemma}[{\cite[Schur's Lemma, p.~18]{kac77}}] \label{lem:Schur}
  Let $\g$ be a Lie superalgebra and $V$ be an irreducible $\g$-module. Define $\End_\g(V)_{i}:=\{T\in \End(V)_{i}\ |\ [T,\g]=0\}$, for $i\in \Z_2$. Then,
  \[
    \End_\g(V)_{\bar{0}}=\C \id,\quad \End_\g(V)_{\bar{1}}=\C \varphi,
  \]
  where $\varphi=0$ or $\varphi^2=- \id$.
\end{lemma}

\subsection{The queer Lie superalgebra} \label{subsec:queer}

Recall the superalgebra $Q(m)$ defined in Section~\ref{subsec:assoc-alg}. If $m=n+1$, then the Lie superalgebra associated to $Q(m)$ will be denoted by $\hat{\q}(n)$.  The derived subalgebra $\tilde{\q}(n)=[\hat{\q}(n),\hat{\q}(n)]$ consists of matrices of the form \eqref{matrixqueer}, where the trace of $B$ is zero. Note that $\tilde{\q}(n)$ has a one-dimensional center spanned by the identity matrix $I_{2n+2}$. The \emph{queer Lie superalgebra} is defined to be the quotient superalgebra
\[
  \q(n)=\tilde{\q}(n)/\C I_{2n+2}.
\]
By abuse of notation, we denote the image in $\q(n)$ of a matrix $X\in \tilde{\q}(n)$ again by $X$. The Lie superalgebra $\q(n)$ has even part isomorphic to $\mathfrak{sl}(n+1)$ and odd part isomorphic (as a module over the even part) to the adjoint module.  One can show that $\q(n)$ is simple for $n \geq 2$ (see \cite[\S 2.4.2]{mus12}). From now on, $\q=\q(n)$ where $n\geq 2$.

\begin{rmk} \label{rem:queer-def}
  Some references refer to $\hat \q(n)$ as the queer Lie superalegbra.  However, in the current paper, we reserve this name for the simple Lie superalgebra $\q(n)$.
\end{rmk}

Denote by $N^-$, $H$, $N^+$ the subset of strictly lower triangular, diagonal, and strictly upper triangular matrices in $\mathfrak{sl} (n+1)$, respectively.  We define
\begin{gather*}
  \h_{\bar{0}}=\left\lbrace
  \left(\begin{array}{r|r}
    A & 0  \\
    \hline
    0 & A \\
  \end{array}\right)
  \ \Big|\ A\in H \right\rbrace, \quad
  \h_{\bar{1}}=\left\lbrace
  \left(\begin{array}{r|r}
    0 & B  \\
    \hline
    B & 0 \\
  \end{array}\right)\ \Big|\ B\in H\right\rbrace, \\
  \mathfrak{n}^{\pm}_{\bar{0}}=\left\lbrace
  \left(\begin{array}{r|r}
    A & 0  \\
    \hline
    0 & A \\
  \end{array}\right)\ \Big|\ A\in N^{\pm}\right\rbrace, \quad
  \mathfrak{n}^{\pm}_{\bar{1}}=
  \left\lbrace  \left(\begin{array}{r|r}
    0 & B  \\
    \hline
    B & 0 \\
  \end{array}\right)\ \Big|\ B\in N^{\pm}\right\rbrace,\\
  \h=\h_{\bar{0}}\oplus \h_{\bar{1}},\quad\text{and}\quad \mathfrak{n}^{\pm}=\mathfrak{n}^{\pm}_{\bar 0}\oplus \mathfrak{n}^{\pm}_ {\bar 1}.
\end{gather*}

\begin{lemma}[{\cite[Lem.~2.4.1]{mus12}}]
  We have a vector space decomposition
  \begin{equation} \label{triangdec}
    \q=\mathfrak{n}^{-}\oplus \h\oplus \mathfrak{n}^{+}
  \end{equation}
  such that $\mathfrak{n}^{\pm}$ and $\h$ are graded subalgebras of $\q$ ,with $\mathfrak{n}^{\pm}$ nilpotent. The subalgebra $\h$ is called the \emph{standard Cartan subalgebra} of $\q$.
\end{lemma}

We now describe the roots of $\q$ with respect to $\h_{\bar 0}$.  For each $i=1,\dotsc, n+1$, define $\epsilon_i\in\h_{\bar 0}^* $ by
\[
  \epsilon_i
  \left(\begin{array}{r|r}
    h & 0  \\
    \hline
    0 & h
  \end{array}\right)=a_i,
\]
where $h$ is the diagonal matrix with entries $(a_1,\dotsc,a_{n+1})$.  For $1\leq i,j\leq n+1$, we let $E_{i,j}$ denote the $(n+1) \times (n+1)$ matrix with a 1 in position $(i,j)$ and zeros elsewhere, and we set
\[
  e_{i,j}=
  \left(\begin{array}{r|r}
    E_{i,j} & 0  \\
    \hline
    0 & E_{i,j}
  \end{array}\right) \quad \text{and} \quad
  e'_{i,j} =
  \left(\begin{array}{r|r}
    0 & E_{i,j}  \\
    \hline
    E_{i,j} & 0
  \end{array}\right).
\]
Given $\alpha\in\h_{\bar 0}^*$, let
\[
  \q_{\alpha}=\{x\in\q\ |\ [h,x]=\alpha(h)x \text{ for all } h\in\h_{\bar{0}}\}.
\]
We call $\alpha$ a \emph{root} if $\alpha \ne 0$ and $\q_\alpha \neq 0$.  Let $\Delta$ denote the set of all roots.  Note that $\q_0=\h$ and, for $\alpha= \epsilon_i-\epsilon_j$, $1\leq i \neq j\leq n+1$, we have
\[
  \q_{\alpha}= \C e_{i,j}\oplus \C e'_{i,j}.
\]
In particular,
\[
  \q = \bigoplus_{\alpha\in \h_{\bar 0}^*}\q_{\alpha}.
\]
A root is called \emph{positive} (resp.\ \emph{negative}) if $\q_{\alpha} \cap \mathfrak{n}^+\neq 0$ (resp.\ $\q_{\alpha}\cap\mathfrak{n}^-\neq 0$).  We denote by $\Delta^+$  (resp.\ $\Delta^-$) the subset of positive (resp.\ negative) roots.  A positive root $\alpha$ is called \emph{simple} if it cannot be expressed as a sum of two positive roots.  We denote by $\Pi$ the set of simple roots.  Thus,
\begin{gather*}
  \Delta^+=\{\epsilon_i-\epsilon_j\ |\ 1\leq i< j\leq n+1\},\quad
  \Pi=\{\epsilon_i-\epsilon_{i+1}\ |\ 1\leq i\leq n+1\}, \\
  \Delta^- = -\Delta^{+},\quad
  \Delta = \Delta^+ \cup \Delta^-.
\end{gather*}
It follows that
\[
  \mathfrak{n}^{+}=\bigoplus_{\alpha\in\Delta^{+}}\q_{\alpha} \quad \text{and} \quad \mathfrak{n}^{-}=\bigoplus_{\alpha\in\Delta^{-}}\q_{\alpha}.
\]
The subalgebra $\mathfrak{b}=\h\oplus \mathfrak{n}^{+}$ is called the \emph{standard Borel subalgebra} of $\q$.

Notice that, since $n\geq 2$, we have $[\h_{\bar 1},\h_{\bar 1}]=\h_{\bar 0}$. Indeed, for all $i,j\in \{1,\dotsc,n+1\}$ with $i\ne j$, we can choose $k\in\{1,\ldots,n+1\}$ such that $k\ne i$, $k\ne j$, and then
\[
  e_{i,i}-e_{j,j}=\frac{1}{2}[e_{i,i}'-e_{j,j}',e_{i,i}'+e_{j,j}'-2e_{k,k}'].
\]
Thus, the result follows from the fact that elements of the form $e_{i,i}-e_{j,j}$, for $i,j\in\{1,\dotsc,n+1\}$ and $i\ne j$, generate $\h_{\bar 0}$.

\subsection{\texorpdfstring{Irreducible finite-dimensional $\q$-modules}{Irreducible finite-dimensional modules for the queer Lie superalgebra}}

\begin{lemma}[{\cite[Prop.~8.2.1]{mus12}}] \label{lem:simple-h-modules}
  For every $\lambda \in \h_{\bar 0}^*$, there exists a unique irreducible finite-dimensional $\h$-module $W(\lambda)$ such that $hv=\lambda(h)v$, for all $h\in \h_{\bar 0}$ and $v\in W(\lambda)$. Moreover any irreducible finite-dimensional $\h$-module is isomorphic to $W(\lambda)$ for some $\lambda\in \h_{\bar 0}^*$.
\end{lemma}

Let $V$ be an irreducible finite-dimensional $\q$-module. For $\mu\in \h_{\bar 0}^*$, let
\[
  V_\mu = \{v\in V\ |\  hv=\mu(h)v \text{ for all } h\in\h_{\bar{0}}\}\subseteq V
\]
be the $\mu$-weight space of $V$.  Since $\h_{\bar 0}$ is an abelian Lie algebra and the dimension of $V$ is finite, we have $V_\mu\neq 0$ for some $\mu\in \h_{\bar 0}^*$.  We also have $\q_\alpha V_\mu\subseteq V_{\mu+\alpha}$, for all $\alpha\in \Delta$.  Then, by the simplicity of $V$, we have the weight space decomposition
\[
  V= \bigoplus_{\mu\in \h_{\bar{0}}^*}V_{\mu}.
\]
Since $V$ has finite dimension, there exists $\lambda \in \h_{\bar 0}^*$ such that $V_\lambda \ne 0$ and $\q_{\alpha}V_\lambda=0$ for all $\alpha\in \Delta^+$.  Since $[\h_{\bar{0}},\h]=0$, each weight space is an $\h$-submodule of $V$. If $W$ is an irreducible $\h$-submodule of $V_\lambda$, then $W\cong W(\lambda)$ by Lemma~\ref{lem:simple-h-modules}.  Now, the irreducibility of $V$ together with the PBW Theorem (Lemma~\ref{lem:PBW}), implies that
\[
  V_\lambda \cong W(\lambda) \quad\text{and} \quad U(\mathfrak{n}^-)V_\lambda=V.
\]
In particular, this shows that any irreducible finite-dimensional $\q$-module is a highest weight module, where the highest weight space is an irreducible $\h$-module.  On the other hand, given an irreducible finite-dimensional $\h$-module $W(\lambda)$, we can consider the Verma type module associated to it. Namely, regard $W(\lambda)$ as a $\mathfrak{b}$-module, where $\n^+W(\lambda)=0$, and consider the induced $\q$-module $U(\q)\otimes_{U(\mathfrak{b})}W(\lambda)$.  This module has a unique proper maximal submodule which we denote by $N(\lambda)$.  Define $V(\lambda) = (U(\q)\otimes_{U(\mathfrak{b})}W(\lambda))/N(\lambda)$.  Then $V(\lambda)$ is an irreducible $\q$-module and every weight of $V(\lambda)$ is of the form
\[
  \lambda -\sum_{\alpha\in \Pi}m_\alpha \alpha,\quad m_\alpha \in \N \text{ for all } \alpha \in \Pi.
\]

\begin{rmk}
  It is important to note here that we allow homomorphisms of modules to be nonhomogeneous.  If we were to require such homomorphisms to be purely even, the Clifford algebra associated to a nondegenerate pair could have two irreducible representations (see Remark~\ref{rmksimpleclif}) and $\mathfrak{q}(n)$ could have two irreducible highest weight representations of a given highest weight.
\end{rmk}

Let $P(\lambda)=\{\mu\in\h_{\bar 0}^*\ |\ V(\lambda)_\mu \ne 0\}$. We will fix the partial order on $P(\lambda)$ given by $\mu_1\geq \mu_2$ if and only if $\mu_1-\mu_2\in Q^+$, where $Q^+ := \sum_{\alpha \in \Pi} \N \alpha$ denotes the positive root lattice of $\q$.

\section{Equivariant map queer Lie superalgebras}\label{Sec:EMQLS}

In this section we introduce our main objects of study: the \emph{equivariant map queer Lie superalgebras}. Henceforth, we let $A$ denote a commutative associative unital algebra and $\q=\q(n)$, with $n \ge 2$.

Consider the Lie superalgebra $\q\otimes A$ where the $\Z_2$-grading is given by $(\q\otimes A)_j=\q_j\otimes A, j\in\Z_{2}$, and the bracket is determined by
\[
  [x_{1}\otimes f_{1},x_{2}\otimes f_{2}]=[x_{1},x_{2}] \otimes f_{1}f_{2},\quad x_{i}\in\q,\ f_{i}\in A,\ i\in \{1,2\}.
\]
We will refer to a Lie superalgebra of this form as a \emph{map queer Lie superalgebra}, inspired by the case where $A$ is the ring of regular functions on an algebraic variety.

\begin{defn}[Equivariant map queer Lie superalgebra]
  Let $\Gamma$ be a group acting on $A$ and on $\q$ by automorphisms.  Then $\Gamma$ acts diagonally on $\q\otimes A$.  We define
  \[
    (\q\otimes A)^\Gamma = \{z \in \q\otimes A\ |\ \gamma z=z \text{ for all } \gamma \in \Gamma\}
  \]
  to be the Lie subalgebra of $\q \otimes A$ of points fixed under the action of $\Gamma$. We call $(\q\otimes A)^\Gamma$ an \emph{equivariant map queer Lie superalgebra}.  Note that if $\Gamma$ is the trivial group, then $(\q\otimes A)^\Gamma=\q\otimes A$.
\end{defn}

\begin{ex}[Multiloop queer superalgebras]
  Let $n,m_1,\dotsc, m_k$ be positive integers and consider the group
  \[
    \Gamma=\langle \gamma_1,\dotsc,\gamma_k\ |\ \gamma_i^{m_i}=1,\ \gamma_i\gamma_j=\gamma_j\gamma_i,\ \forall\ 1\leq i,j\leq k\rangle \cong \bigoplus_{i=1}^k\Z/m_i\Z.
  \]
  An action of $\Gamma$ on $\q$ is equivalent to a choice of commuting automorphisms $\sigma_i$ of $\q$ such that $\sigma_i^{m_i}=\id$, for all $i=1,\dotsc, n$. Let $A=\C[t_1^{\pm },\dotsc,t_k^{\pm}]$ be the algebra of Laurent polynomials in $k$ variables and let $X = \Spec(A)$ (in other words, $X$ is the $k$-torus $(\C^\times)^k$). For each $i=1,\dotsc,k$, let $\xi_i\in \C$ be a primitive $m_i$-th root of unity, and define an action of $\Gamma$ on $X$ by
  \[
    \gamma_i(z_1,\dotsc,z_k)=(z_1,\dotsc,z_{i-1},\xi_iz_i,z_{i+1},\dotsc,z_k).
  \]
  This induces an action on $A$ and we call
  \begin{equation*}\label{llopsup}
  M(\q,\sigma_1,\dotsc,\sigma_k,m_1,\dotsc,m_k) := (\q\otimes A)^\Gamma
  \end{equation*}
  the (\emph{twisted}) \emph{multiloop queer superalgebra} relative to $(\sigma_1,\dotsc,\sigma_k)$ and $(m_1,\dotsc, m_k)$. If $\Gamma$ is trivial, we call it an \emph{untwisted multiloop queer superalgebra}.  If $n=1$, then $M(\q,\sigma_1,m_1)$ is called a (\emph{twisted} or \emph{untwisted}) \emph{loop queer superalgebra}.  These have been classified, up to isomorphism, in \cite[Th.~4.4]{GP04}.  This classification uses the fact that the outer automorphism group of $\q$ is isomorphic to $\Z_4$ (see \cite[Th.~1]{Ser84}).
\end{ex}

\begin{defn}[$\Ann_A(V)$, $\Supp(V)$]
  Let $V$ be a $(\q\otimes A)^\Gamma$-module.  We define $\Ann_A(V)$ to be the sum of all $\Gamma$-invariant ideals $I \subseteq A$, such that $(\q\otimes I)^\Gamma V=0$.  If $\rho$ is the associated representation, we set $\Ann_A(\rho):=\Ann_A(V)$.  We define the \emph{support} of $V$ to be the support of $\Ann_A(V)$ (see Definition~\ref{def:ideal-support}).  We say that $V$ has \emph{reduced support} if $\Ann_A(V)$ is a radical ideal.
\end{defn}

\section{Irreducible finite-dimensional representations of the Cartan subalgebra}
\label{sec:cartan-modules}

In this section we study irreducible finite-dimensional $\h\otimes A$-modules. The goal is to show that, for each such module, there exists a finite-codimensional ideal $I\subseteq A$, such that $I$ is maximal with respect to the property $(\h\otimes I)V=0$. Once this is done, we can proceed using similar arguments to those used in the study of irreducible finite-dimensional $\h$-modules (see \cite[Prop.~8.2.1]{mus12} or  \cite[\S 1.5.4]{chewan12} for example).

\begin{lemma} \label{lem:even-odd-ideal-kill}
  Let $V$ be an irreducible finite-dimensional $\h\otimes A$-module and let $I\subseteq A$ be an ideal such that $(\h_{\bar{0}}\otimes I)V=0$. Then $(\h_{\bar 1}\otimes I)V=0$.
\end{lemma}

\begin{proof}
  Let $\rho$ be the associated representation of $\h\otimes A$ on $V$.  We must prove that $\rho(\h_{\bar 1}\otimes I)=0$. Note that
  \[
    [\rho(\h \otimes A), \rho(\h_{\bar 1} \otimes I)] = \rho([\h\otimes A,\h_{\bar 1}\otimes I])\subseteq \rho([\h,\h_{\bar 1}] \otimes I)\subseteq \rho(\h_{\bar 0}\otimes I)=0.
  \]
  Thus, $\rho(\h_{\bar 1}\otimes I)\subseteq \End_{\h\otimes A}(V)_{\bar{1}}$.  Suppose that $\rho(z) \ne 0$ for some $z \in \h_{\bar 1} \otimes I$.  Then, possibly after multiplying $z$ by a nonzero scalar, we may assume, by Schur's Lemma (Lemma~\ref{lem:Schur}), that $\rho(z)^2 = -\id$.  But then we obtain the contradiction
  \[
    -2\id = 2\rho(z)^2=[\rho(z),\rho(z)]=\rho([z,z])=0,
  \]
  where the last equality follows from the fact that $[z,z] \in \h_{\bar 0} \otimes I$.
\end{proof}

\begin{prop}\label{prop:finitecod}
  Let $V$ be an irreducible $\h\otimes A$-module. Then $V$ is finite-dimensional if and only if there exists a finite-codimensional ideal $I$ of $A$ such that $(\h\otimes I)V=0$.
\end{prop}

\begin{proof}
  Suppose $V$ is an irreducible finite-dimensional $\h\otimes A$-module, and let $\rho$ be the associated representation. Let $I$ be the kernel of the linear map
  \[
    \varphi: A\rightarrow \Hom_\C(V\otimes \h,V),\quad a\mapsto (v\otimes h\mapsto\rho (h\otimes a)v),\quad a\in A,\ v\in V,\ h\in \h.
  \]
   Since $V$ is finite-dimensional, $I$ is a linear subspace of $A$ of finite-codimension. We claim that $I$ is an ideal of A. Indeed, if  $r\in A$, $a\in I$ and $v\in V$, then we have
  \begin{multline*}
    \varphi(ra)(v\otimes \h_{\bar 0})=\rho (\h_{\bar 0}\otimes ra)v=\rho([\h_{\bar 1},\h_{\bar 1}]\otimes ra)v \\
    = \rho([\h_{\bar 1}\otimes r,\h_{\bar 1}\otimes a])v=\rho(\h_{\bar 1}\otimes r)\rho(\h_{\bar 1}\otimes a)v+\rho(\h_{\bar 1}\otimes a)\rho(\h_{\bar 1}\otimes r)v=0.
  \end{multline*}
  Thus $\varphi(ra)(V\otimes \h_{\bar 0})=0$ for all $r \in A$, $a \in I$, or equivalently, $\rho(\h_{\bar 0}\otimes A I)=0$. In particular,
  \[
    [\rho(\h_{\bar 1}\otimes A I),\rho(\h\otimes A)]\subseteq \rho(\h_{\bar 0}\otimes AI)=0,
  \]
  which implies that $\rho(\h_{\bar 1}\otimes A I)\subseteq \End_{\h\otimes A}(V)_{\bar{1}}$.  Suppose now that $\varphi(ra)(v\otimes h)\ne 0$ for some $v \in V$ and $h\in\h_{\bar 1}$. Then $\rho(h\otimes ra)\ne 0$, with $h\otimes ra\in \h_{\bar 1}\otimes AI $. Thus, as in the proof of Lemma~\ref{lem:even-odd-ideal-kill}, we are lead to the contradiction (possibly after rescaling $h\otimes ra$):
  \[
    -2 \id=2\rho(h\otimes ra)^2=[\rho(h\otimes ra),\rho(h\otimes ra)]\in \rho(\h_{\bar 0}\otimes (rar)a)=0,
  \]
  where, in the last equality, we used that $\rho(\h_{\bar 0}\otimes A I)=0$. Since $V \otimes \h_{\bar 1}$ is spanned by simple tensors of the form $v \otimes h$, $v \in V$, $h \in \h_{\bar 1}$, it follows that $\varphi(ra)(V\otimes \h_{\bar 1})=0$, and so $ra \in I$.  Thus $I$ is a finite-codimensional ideal of $A$ such that $(\h\otimes I)V=0$.

  Conversely, suppose that $(\h\otimes I)V=0$ for some ideal $I\subseteq A$ of finite codimension. Then $V$ factors to an irreducible $\h\otimes A/I$-module with $(\h_{\bar 0}\otimes A/I)v\subseteq \C v$ for all $v\in V$ by Schur's Lemma (Lemma~\ref{lem:Schur}). On the other hand, let  $\{x_1,\dotsc,x_k\}$ be a basis of $\h_{\bar{1}}\otimes A/I$. Since $V$ is irreducible, the PBW Theorem (Lemma~\ref{lem:PBW}) implies that
  \[
    V= U(\h\otimes A/I)v=\sum_{1\leq i_1<\cdots<i_s\leq k}x_{i_1}\cdots x_{i_s}\C v,
  \]
  where $i_1,\dotsc,i_s \in \{1,\dotsc,k\}$.  Hence, $V$ is finite-dimensional.
\end{proof}

Let
\[
  \mathcal{L}(\h\otimes A)=\{\psi\in (\h_{\bar{0}}\otimes A)^*\ |\ \psi(\h_{\bar{0}}\otimes I)=0, \text{ for some finite-codimensional ideal } I\subseteq A\}
\]
and let $\mathcal{R}(\h\otimes A)$ denote the set of isomorphism classes of irreducible finite-dimensional $\h\otimes A$-modules. If $\psi\in \mathcal{L}(\h\otimes A)$ and $S=\{I\subseteq A\ |\ I\text{ is an ideal, and } \psi(\h_{\bar{0}}\otimes I)=0\}$, we set $I_\psi=\sum_{I\in S}I$.

\begin{thm}\label{class.cartan.mod}
  For any $\psi\in \mathcal{L}(\h\otimes A)$, there exists a unique, up to isomorphism, irreducible finite-dimensional $\h\otimes A$-module $H(\psi)$ such that $x v=\psi(x)v$, for all $x\in \h_{\bar 0}\otimes A$ and $v \in H(\psi)$. Conversely, any irreducible finite-dimensional $\h\otimes A$-module is isomorphic to $H(\psi)$, for some $\psi\in \mathcal{L}(\h\otimes A)$. In other words, the map
  \[
    \mathcal{L}(\h\otimes A) \rightarrow \mathcal{R}(\h\otimes A),\quad \psi\mapsto H(\psi),
  \]
  is a bijection.
\end{thm}

\begin{proof}
  Assume first that $V$ is an irreducible finite-dimensional $\h\otimes A$-module and that $x v=0$ for all $x\in \h_{\bar 0}\otimes A$ and $v \in V$.  Then, by Lemma~\ref{lem:even-odd-ideal-kill}, we have $(\h\otimes A)V=0$.  So we take $H(0)$ to be the trivial module.

  Assume now $\psi\in \mathcal{L}(\h\otimes A)$ and $\psi\neq 0$. Define a symmetric bilinear form $f_\psi$ on $\h_{\psi}:=\h_{\bar{1}}\otimes A/I_\psi$ by
  \begin{equation}\label{bilform}
    f_\psi (x,y)=\psi([x,y]),\quad x, y \in \h_{\psi}.
  \end{equation}
  Let $\h_{\psi}^\perp = \{x\in \h_\psi\ |\ f_\psi(x,\h_\psi)=0\}$ denote the radical of $f_\psi$, and set
  \[
    \mathfrak{c}_\psi:=\frac{\h\otimes A/I_\psi}{(\ker \psi)\oplus \h_{\psi}^\perp}\cong \frac{ \h_{\bar 0}\otimes A/I_\psi}{\ker \psi} \oplus \frac{\h_\psi}{\h_\psi^\perp}.
  \]

  We can regard $\psi$ as a linear functional on $(\mathfrak{c}_\psi)_{\bar 0}$. Since $\psi\neq 0$, and $\dim (({\frak c}_\psi)_{\bar 0})=1$, there exists a unique $z\in (\mathfrak{c}_\psi)_{\bar 0}$ such that $\psi(z)=1$. Define the factor algebra $A_\psi:=U(\mathfrak{c}_\psi)/(z-1)$.  Consider the natural linear maps $i \colon (\mathfrak{c}_\psi)_{\bar 1} \hookrightarrow T((\mathfrak{c}_\psi)_{\bar 1})$ and $p \colon T((\mathfrak{c}_\psi)_{\bar 1}) \twoheadrightarrow A_\psi$.  It is straightforward to check, via the universal property of Clifford algebras (see Remark~\ref{univ-cliff-prop}), that the pair $(A_\psi,p\circ i)$ is isomorphic to the Clifford algebra of $((\mathfrak{c}_\psi)_{\bar 1},\frac{1}{2}f_\psi)$.  By Remark~\ref{rmksimpleclif}, this Clifford algebra admits only one, up to isomorphism, irreducible finite-dimensional module. Let $H(\psi)$ denote such a module. We can consider an action of $\h\otimes A$ on $H(\psi)$ via the map
  \[
    \h\otimes A\twoheadrightarrow \mathfrak{c}_\psi\hookrightarrow U(\mathfrak{c}_\psi) \twoheadrightarrow  A_\psi.
  \]
  Note that $H(\psi)$ is an irreducible $U(\mathfrak{c}_\psi)$-module (and thus an irreducible $\mathfrak{c}_\psi$-module).  Therefore, $H(\psi)$ is an irreducible finite-dimensional $\h\otimes A$-module. In particular we have that
  \[
    xv=\psi(x)v, \text{ for all } x\in\h_{\bar{0}}\otimes A \text{ and }  v\in H(\psi).
  \]

  It remains to prove the converse statement in the lemma.  Let $V$ be any irreducible finite-dimensional $\h\otimes A$-module with associated representation $\rho$.  Since $\h_{\bar{0}}\otimes A$ is central in $\h\otimes A$, there exists $\psi\in (\h_{\bar{0}}\otimes A)^*$ such that $x v=\psi(x)v$, for all $x\in \h_{\bar{0}}\otimes A$, $v\in V$. On the other hand, by Proposition~\ref{prop:finitecod}, there exists an ideal $I$ of $A$ of finite-codimension such that $(\h\otimes I)V=0$. In particular, we have that $\psi(\h_{\bar{0}}\otimes I)=0$, so $\psi\in \mathcal{L}(\h\otimes A)$, and that $V$ factors to an irreducible $\h\otimes A/I_\psi$-module. If $\h_{\psi}^\perp$ is defined to be the radical of the bilinear form~\eqref{bilform}, then $\rho(\h_{\psi}^\perp)\subseteq \gl (V)$ is central. Since $\rho$ is irreducible and $\rho(\h_{\psi}^{\perp})$ consists of odd elements, it follows that $\rho(\h_{\psi}^{\perp})=0$. Hence, $V$ is an irreducible finite-dimensional $C((\mathfrak{c}_\psi)_{\bar 1},\frac{1}{2}f_\psi)$-module, and so $V \cong H(\psi)$.
\end{proof}

\section{Highest weight modules} \label{sec:hw-modules}

In \cite{sav14}, the irreducible finite-dimensional modules of an equivariant map Lie superalgebra were investigated in the case that the target Lie superalgebra is basic classical.  In particular, it was proved that such modules are either generalized evaluation modules or quotients of analogues of Kac modules of some evaluation modules for a reductive Lie algebra.  It was heavily used that the highest weight space of any irreducible finite-dimensional module is one-dimensional, and also that tensor products of irreducible modules with disjoint supports are again irreducible modules.

In Section~\ref{sec:cartan-modules}, we saw that irreducible finite-dimensional modules for the Cartan superalgebra $\h \otimes A$ are irreducible modules for certain Clifford algebras.  In particular, the dimension of such modules is not necessarily equal to one.  In addition, it is not true, in general, that the tensor product of irreducible modules with disjoint supports is irreducible (see Example~\ref{eg:nonirred}).  Thus, the arguments used in \cite{sav14} require modification.

From now on, we consider $\q \subseteq \q \otimes A$ as a Lie subalgebra via the natural isomorphism $\q \cong \q \otimes \Bbb C$.  We also fix the triangular decomposition of $\q$ given in \eqref{triangdec}.

\begin{defn}[Weight module]
  Let $V$ be a $\q\otimes A$-module. We say $V$ is a \emph{weight module} if it is a sum of its weight spaces, i.e.,
  \[
    V=\bigoplus_{\lambda\in\h^{*}}V_{\lambda}, \quad \text{where } V_{\lambda}=\{ v\in V \ |\ h v=\lambda(h)v \text{ for all } h \in \h_{\bar{0}} \}.
  \]
  If $V_\lambda\neq 0$, then $\lambda \in \h_{\bar 0}^*$ is called a \emph{weight} of $V$ and the nonzero elements of $V_\lambda$ are called \emph{weight vectors} of weight $\lambda$.
\end{defn}

\begin{defn}[Quasifinite module]
  A weight $\q\otimes A$-module is called \emph{quasifinite} if all its weight spaces are finite-dimensional.
\end{defn}

\begin{defn}[Highest weight module] \label{def:hw-module}
  A $\q\otimes A$-module $V$ is called a \emph{highest weight module} if there exists a nonzero vector $v\in V$ such that
  \begin{equation} \label{mpeso}
    U(\q\otimes A)v=V,\quad (\mathfrak{n}^{+}\otimes A)v=0, \quad\text{and}\quad U(\h_{\bar 0}\otimes A) v = \C v.
  \end{equation}
  We call $v$ a \emph{highest weight vector}.
\end{defn}

\begin{prop} \label{prop:irred-module-hw}
  If $V$ is an irreducible finite-dimensional $\q\otimes A$-module, then $V$ is a highest weight module. Moreover, the weight space associated to the highest weight is an irreducible $\h\otimes A$-module.
\end{prop}

\begin{proof}
  Since $\h_{\bar 0}$ is an abelian Lie algebra and the dimension of $V$ is finite, $V_\mu\neq 0$ for some $\mu\in \h_{\bar 0}^*$. Also note that $(\q_\alpha\otimes A) V_\mu\subseteq V_{\mu+\alpha}$, for all $\alpha\in \Delta$. Then, by the simplicity of $V$, it is a weight module. Since $V$ is finite-dimensional, there exists a maximal weight $\lambda\in\h_{\bar{0}}^*$, such that $V_{\lambda}\neq 0$.  It follows immediately that
  \[
    (\mathfrak{n}^+\otimes A)V_{\lambda}=0.
  \]
  Considering $V_{\lambda}$ as an $\h\otimes A$-module, we can choose an irreducible $\mathfrak{h} \otimes A$-submodule $H(\psi) \subseteq V_\lambda$. Thus $U(\h_{\bar 0}\otimes A)v\subseteq \C v$, for all $v\in H(\psi)$. Now by the simplicity of $V$, we have $U(\q\otimes A)v=V$ for any $v\in H(\psi)$. In particular, the PBW Theorem (Lemma~\ref{lem:PBW}) implies that $V_{\lambda}=H(\psi)$.
\end{proof}

Fix $\psi\in \mathcal{L}(\h\otimes A)$ and define an action of $\mathfrak{b}\otimes A$ on $H(\psi)$ by declaring $\mathfrak{n}^+\otimes A$ to act by zero. Consider the induced module
\[
  \bar{V}(\psi)=U(\q\otimes A)\otimes_{U(\mathfrak{b}\otimes A)}H(\psi).
\]
This is a highest weight module, and a submodule of $\bar{V}(\psi)$ is proper if and only if its intersection with $H(\psi)$ is zero. Moreover any $\q\otimes A$-submodule of a weight module is also a weight module. Hence, if $W\subseteq \bar{V}(\psi)$ is a proper $\q\otimes A$-submodule, then
\[
  W=\bigoplus_{\mu\neq \lambda} W_{\mu},
\]
where $\lambda=\psi|_{\h_{\bar{0}}}$. Therefore $\bar{V}(\psi)$ has a unique maximal proper submodule $N(\psi)$ and
\[
  V(\psi)=\bar{V}(\psi)/N(\psi)
\]
is an irreducible highest weight $\q\otimes A$-module.

By Proposition~\ref{prop:irred-module-hw}, every irreducible finite-dimensional $\q\otimes A$-module is isomorphic to $V(\psi)$ for some $\psi\in \mathcal{L}(\h\otimes A)$. Notice also that the highest weight space of $V(\psi)$ is isomorphic, as an $\h \otimes A$-module, to $H(\psi)$.

\begin{lemma} \label{lem:psi-0-ideal}
  Let $\psi\in \mathcal{L}(\h\otimes A)$ and let $I$ be an ideal of $A$.  Then $\psi(\h_{\bar{0}}\otimes I)=0$ if and only if $(\q\otimes I)V(\psi)=0$.
\end{lemma}

\begin{proof}
  Suppose that $\psi(\h_{\bar{0}}\otimes I)=0$ and set $\lambda=\psi|_{\h_{\bar 0}}$. We know that $V(\psi)_{\lambda} \cong H(\psi)$ as $\h \otimes A$-modules, and, by Lemma~\ref{lem:even-odd-ideal-kill}, we have that $(\h\otimes I)V(\psi)_{\lambda}=0$. Now, let $v$ be a nonzero vector in $V(\psi)_{\lambda}$. By Lemma~\ref{lem:ideal-annihilate-vector}, to prove that $(\q\otimes I)V(\psi)=0$, it is enough to prove that $(\q\otimes I)v=0$.  It is clear that $(\h\otimes I)v=0$ and, since $v$ is a highest weight vector, we also have that $(\mathfrak{n}^{+}\otimes I)v=0$.  It remains to show that $(\mathfrak{n}^{-}\otimes I)v=0$.

  For $\alpha=\sum_{i=1}^n a_i\alpha_i$, with $a_i \in \N$ and where the $\alpha_i$ are the simple roots of $\q$, we define the \emph{height} of $\alpha$ to be
  \[
    \hgt(\alpha)=\sum_{i=1}^n a_i.
  \]
  By induction on the height of $\alpha$, we will show that $(\q_{-\alpha}\otimes I)v=0$. We already have the result for $\hgt(\alpha)=0$ (since $\q_{0}=\h$). Suppose that, for some $m\geq 0$, the results holds whenever $\hgt(\alpha)\leq m$. Fix $\alpha\in \Delta^+$ with $\hgt(\alpha)=m+1$. Then
  \begin{equation}\label{e:annihilator}
      (\mathfrak{n}^+\otimes A)(\q_{-\alpha}\otimes I)v \subseteq [\mathfrak{n}^+\otimes A,\q_{-\alpha}\otimes I]v+(\q_{-\alpha}\otimes I)(\mathfrak{n}^+\otimes A)v
      =  ([\mathfrak{n}^+,\q_{-\alpha}]\otimes I)v=0,
  \end{equation}
  where the last equality follows from the induction hypothesis, since any element of $[\mathfrak{n}^+,\q_{-\alpha}]$ is either an element of $\q_{-\gamma}$, with $\hgt(\gamma)<\hgt(\alpha)$, or an element of $\mathfrak{n}^+$.  Now suppose that there exists a nonzero vector $w\in(\q_{-\alpha}\otimes I)v\subseteq V_{\lambda-\alpha}$. By \eqref{e:annihilator}, we have $(\n^+\otimes A)w=0$, and, since $V$ is irreducible, we have $V=U(\q\otimes A)w$. But, by the PBW Theorem (Lemma~\ref{lem:PBW}), this implies that $V(\psi)_\lambda= 0$, which is a contradiction. This completes the proof of the forward implication.  The reverse implication is obvious.
\end{proof}

\begin{thm} \label{thm:quasifinite}
  Let $\psi\in \mathcal{L}(\h\otimes A)$. The following conditions are equivalent:
  \begin{enumerate}
    \item \label{thm-item:quasifinite} The module $V(\psi)$ is quasifinite.
    \item \label{thm-item:ideal-kill} There exists a finite-codimensional ideal $I$ of $A$ such that $(\q\otimes I)V(\psi)=0$.
    \item \label{thm-item:ideal-h-kill} There exists a finite-codimensional ideal $I$ of $A$ such that $\psi(\h_{\bar{0}}\otimes I)=0$.
  \suspend{enumerate}
  If $A$ is finitely generated, then the above conditions are also equivalent to:
  \resume{enumerate}
    \item \label{thm-item:finite-support} The module $V(\psi)$ has finite support.
  \end{enumerate}
\end{thm}

\begin{proof}
  $\eqref{thm-item:quasifinite}\Rightarrow\eqref{thm-item:ideal-kill}$: Let $\lambda=\psi|_{\h_{\bar{0}}}$ be the highest weight of $V(\psi)$. Let $\alpha$ be a positive root of $\q$ and let $I_{\alpha}$ be the kernel of the linear map
  \[
    A\rightarrow \Hom_\C(V(\psi)_{\lambda}\otimes \q_{-\alpha},V(\psi)_{\lambda-\alpha}),\quad f\mapsto (v\otimes u\mapsto(u\otimes f)v),\quad f\in A,\ v\in V(\psi)_{\lambda},\ u\in \q_{-\alpha}.
  \]
  Since $V(\psi)$ is quasifinite, $I_{\alpha}$ is a linear subspace of $A$ of finite-codimension. We claim that $I_{\alpha}$ is, in fact, an ideal of A. Indeed, since $\alpha\neq 0$, we can choose $h\in\h_{\bar{0}}$ such that $\alpha(h)\neq 0$. Then, for all $g\in A$, $f\in I_{\alpha}$, $v\in V(\psi)_{\lambda}$ and $u\in \q_{-\alpha}$, we have
  \begin{eqnarray}
  0 & = & (h\otimes g)(u\otimes f)v \nonumber \\
     & = & [h\otimes g,u\otimes f]v + (u\otimes f)(h\otimes g)v \nonumber \\
     & = & -\alpha(h)(u\otimes gf)v + (u\otimes f)(h\otimes g)v. \nonumber
  \end{eqnarray}
  Since $(h\otimes g)v\in V(\psi)_{\lambda}$ and $f\in I_{\alpha}$, the last term above is zero. Since we also have $\alpha(h)\neq 0$, this implies that $(u\otimes gf)v=0$. As this holds for all $v\in V(\psi)_{\lambda}$ and $u\in \q_{-\alpha}$, we have $gf\in I_{\alpha}$. Hence $I_{\alpha}$ is an ideal of $A$.

  Let $I$ be the intersection of all the $I_{\alpha}$. Since $\q$ is finite-dimensional (and thus has a finite number of positive roots), this intersection is finite and thus $I$ is also an ideal of $A$ of finite-codimension. We then have $(\mathfrak{n}^{-}\otimes I)V(\psi)_{\lambda}=0$. Since $\lambda$ is the highest weight of $V(\psi)$, we also have $(\mathfrak{n}^{+}\otimes I)V(\psi)_{\lambda}=0$. Then, since $\h\otimes I\subseteq [\mathfrak{n}^{+}\otimes A,\mathfrak{n}^{-}\otimes I]$, we have $(\h\otimes I)V(\psi)_{\lambda}=0$.  Thus $(\q\otimes I)V(\psi)_{\lambda}=0$. Since $V(\psi)_{\lambda}\neq 0$, it follows from Lemma~\ref{lem:ideal-annihilate-vector} that $(\q\otimes I)V(\psi)=0$.

  $\eqref{thm-item:ideal-kill}\Rightarrow\eqref{thm-item:ideal-h-kill}$: Let $v$ be a highest weight vector of $V(\psi)$. Then $\psi(x)v=x v=0$, for any $x\in \h_{\bar{0}}\otimes I$. Thus $\psi(\h_{\bar{0}}\otimes I)=0$.

  $\eqref{thm-item:ideal-h-kill}\Rightarrow\eqref{thm-item:quasifinite}$: If $\psi(\h_{\bar 0}\otimes I)=0$, then, by Lemma~\ref{lem:psi-0-ideal}, we have $(\q \otimes I)V(\psi)=0$.  Then $V(\psi)$ is naturally a module for the finite-dimensional Lie superalgebra $\q\otimes A/I$.  By the PBW Theorem (Lemma~\ref{lem:PBW}), we have
  \[
    V(\psi) = U(\q\otimes A/I)V(\psi)_\lambda  =U(\mathfrak{n}^- \otimes A/I)V(\psi)_\lambda,
  \]
  and $V(\psi)_\lambda$ is finite-dimensional. Another standard application of the PBW Theorem completes the proof.

  Now suppose $A$ is finitely generated. We prove that $\eqref{thm-item:ideal-kill}\Leftrightarrow \eqref{thm-item:finite-support}$. By definition, $\Supp_{A}(V(\psi))=\Supp(\Ann_{A}(V(\psi)))$, where $\Ann_{A}(V(\psi))$ is the largest ideal of $A$ such that $(\q\otimes I)V(\psi)=0$.  Thus $\eqref{thm-item:ideal-kill}$ is true if and only of $\Ann_{A}(V(\psi))$ is of finite-codimension. Since $A$ is finitely generated, $\Ann_{A}(V(\psi))$ is of finite-codimension if and only if it has finite support (see Lemma~\ref{lem:assoc-alg-facts}, parts~\eqref{lem-item:finite-codim-finite} and~\eqref{lem-item:finite-finite-codim}).
\end{proof}

\begin{cor}\label{cor:fd-finite-codim-ideal}
  Let $V$ be an irreducible finite-dimensional $\q\otimes A$-module. Then, there exists an ideal $I$ of $A$ of finite-codimension such that $(\q\otimes I)V=0$.
\end{cor}

\begin{proof}
  Since finite-dimensional modules are clearly quasifinite, the result follows from Theorem~\ref{thm:quasifinite}.
\end{proof}

\section{Evaluation representations and their irreducible products} \label{sec:evaluation}

If $R$ and $S$ are associative unital algebras, all irreducible finite-dimensional modules for $R \otimes S$ are of the form $V_{R}\otimes V_{S}$, where $V_R$ and $V_S$ are irreducible finite-dimensional modules for $R$ and $S$, respectively.  Furthermore, all such modules are irreducible.  When $R$ and $S$ are allowed to be superalgebras, the situation is somewhat different.  In particular, $V_{R}\otimes V_{S}$ is not necessarily irreducible, as seen in the next example.

\begin{ex} \label{eg:nonirred}
  By Remark~\ref{rmksimpleclif}, the unique irreducible finite-dimensional $Q(1)$-module is $\C^{1|1}$. However, $\C^{1|1}\otimes \C^{1|1}$ is not an irreducible $Q(1)\otimes Q(1)$-module, since $Q(1)\otimes Q(1)\cong M(1|1)$ as associative superalgebras and, again by Remark~\ref{rmksimpleclif}, the unique irreducible finite-dimensional $M(1|1)$-module is also $\C^{1|1}$.
\end{ex}

In general, if $\g^1$, $\g^2$ are two finite-dimensional Lie superalgebras, and $V^i$ is an irreducible finite-dimensional $\g^i$-module, for $i=1,2$, then the $\g^1\oplus \g^2$-module $V^1\otimes V^2$ is irreducible only if $\End_{\g^i} (V^i)_{\bar 1}=0$, for some $i=1,2$ (recall that $\dim (\End_{\g^i} (V^i)_{\bar 1})\ne 0$ implies, by Schur's Lemma (Lemma~\ref{lem:Schur}), that $\End_{\g^i} (V^i)_{\bar 1}= \C \varphi_i$, where $\varphi_i^2=-1$). When $\End_{\g^i} (V^i)_{\bar 1} = \C\varphi_i$, $\varphi_i^2=-1$, for $i=1$ and $i=2$, we have that
\[
  \widehat{V}=\{v\in V^1\otimes V^2\ |\ (\tilde{\varphi}_1\otimes \varphi_2)v=v\},\text{ where } \tilde{\varphi}_1=\sqrt{-1}\,\varphi_1,
\]
is an irreducible $\g^1\oplus\g^2$-submodule of $V^1\otimes V^2$ such that  $V^1\otimes V^2\cong \widehat{V}\oplus \widehat{V}$ (see \cite[p.~27]{che95}). Set henceforth
\begin{equation}\label{h-tensoreq}
  V^1\widehat{\otimes} V^2 =
  \begin{cases}
    V^1\otimes V^2 & \text{if } V^1\otimes V^2 \text{ is irreducible}, \\
    \widehat{V}\varsubsetneq V^1\otimes V^2 & \text{if } V^1\otimes V^2 \text{ is not irreducible}.
  \end{cases}
\end{equation}

In \cite[Prop.~8.4]{che95}, it is proved that every irreducible finite-dimensional $\g^1\oplus\g^2$-module is isomorphic to a module of the form $V^1\widehat{\otimes} V^2$, where $V^i$ is an irreducible finite-dimensional $\g^i$-module for $i=1,2$. If $\rho_i$ denotes the representation associated to the $\g^i$-module $V^i$, then $\rho_1{\widehat \otimes}\rho_2$ will denote the representation associated to the $\g^1\oplus\g^2$-module  $V^1{\widehat \otimes }V^2$. Inductively, we define the $\g^1\oplus\cdots\oplus\g^k$-module
\[
  V^1{\widehat \otimes }\cdots {\widehat \otimes }V^k:=(V^1{\widehat \otimes }\cdots {\widehat \otimes} V^{k-1}){\widehat \otimes } V^k
\]
with associated representation denoted by $\rho_1{\widehat\otimes}\cdots {\widehat \otimes }\rho_k$. We will call $\widehat{\otimes}$ the \emph{irreducible product}.  As the next lemma shows, it is associative, up to isomorphism.

\begin{lemma}\label{ass.lemma}
  For $i=1,2,3$, let $\g^i$ be a Lie superalgebra and let $V^i$ be an irreducible finite-dimensional $\g^i$-module. Then, $(V^1\widehat{\otimes}V^2)\widehat{\otimes} V^3\cong V^1\widehat{\otimes}(V^2\widehat{\otimes} V^3)$ as $\g^1\oplus \g^2\oplus \g^3$-modules.
\end{lemma}

\begin{proof}
  By \cite[Prop.~8.4]{che95}, the unique, up to isomorphism, irreducible finite-dimensional $\g^1 \oplus(\g^2 \oplus \g^3)$-module contained in $V^1 \otimes(V^2 \otimes V^3)$ is $V^1 \widehat{\otimes}(V^2 \widehat{\otimes} V^3)$. On the other hand, the unique irreducible finite-dimensional $(\g^1 \oplus\g^2)\oplus \g^3$-module contained in $(V^1\otimes V^2)\otimes V^3$ is $(V^1\widehat{\otimes} V^2)\widehat{\otimes} V^3$. Now, since $\g^1\oplus \g^2\oplus \g^3\cong\g^1\oplus(\g^2\oplus \g^3)\cong (\g^1\oplus\g^2)\oplus \g^3$ as Lie superalgebras, and $(V^1\otimes V^2)\otimes V^3\cong V^1\otimes (V^2 \otimes V^3)$ as $\g^1\oplus \g^2\oplus \g^3$-modules, we conclude that $(V^1\widehat{\otimes}V^2)\widehat{\otimes} V^3\cong V^1\widehat{\otimes}(V^2\widehat{\otimes} V^3)$.
\end{proof}

\begin{prop} \label{prop:irred-tensor}
  Let $V(\psi_1)$ and $V(\psi_2)$, for $\psi_1, \psi_2 \in \mathcal{L}(\h\otimes A)$, be two irreducible finite-dimensional $\g\otimes A$-modules with disjoint supports. Then
  \[
    V(\psi_1)\otimes V(\psi_2) \cong
    \begin{cases}
      V(\psi_1+\psi_2), \text{ or} \\
      V(\psi_1+\psi_2)\oplus V(\psi_1+\psi_2).
    \end{cases}
  \]
\end{prop}

\begin{proof}
  Let $I_i=\text{Ann}_A(V(\psi_i))$ and let $\rho_i$ be the representation corresponding to $V(\psi_i)$, for $i=1,2$. Then the representation $\rho_1\otimes \rho_2$ factors through the composition
  \begin{equation} \label{eq:irred-tensor-factor}
    \q\otimes A \stackrel{\Delta}{\hookrightarrow} (\q\otimes A)\oplus (\q\otimes A)\stackrel{\pi}{\twoheadrightarrow} (\q\otimes A/I_{1})\oplus (\q\otimes A/I_{2}),
  \end{equation}
  where $\Delta$ is the diagonal embedding and the second map is the obvious projection on each summand. By Lemma~\ref{lem:assoc-alg-facts}\eqref{lem-item:product-intersection}, we have that $I_1\cap I_2=I_1 I_2$, since the supports of $I_1$ and $I_2$ are disjoint.  Thus $A=I_1+I_2$, and so $A/I_1 I_2\cong (A/I_1)\oplus (A/I_2)$. We therefore have the following commutative diagram:
  \[
    \xymatrix{\q\otimes A \ar@{->>}[d]  \ar@{^{(}->}[r]^-{\Delta} & (\q\otimes A)\oplus (\q\otimes A) \ar@{->>}[d] \\
    \q\otimes A/I_{1}I_{2} \ar[r]^-{\cong} & (\q\otimes A/I_{1})\oplus (\q\otimes A/I_{2})}
  \]

  It follows that the composition~\eqref{eq:irred-tensor-factor} is surjective.  However, as a $(\q\otimes A/I_1)\oplus (\q\otimes A/I_2)$-module, $V(\psi_1)\otimes V(\psi_2)$ is either irreducible or is isomorphic to $\widehat{V}\oplus \widehat{V}$, where $\widehat{V}\varsubsetneq V(\psi_1)\otimes V(\psi_2)$ is an irreducible $(\q\otimes A/I_1)\oplus (\q\otimes A/I_2)$-module. Then the result follows from the fact that $V(\psi_1)\otimes V(\psi_2)$, and hence ${\widehat V}\oplus {\widehat V}$, is generated by vectors on which $\h \otimes A$ acts by $\psi_1 + \psi_2$.
\end{proof}

Note that if $V(\psi_1)$ and $V(\psi)$ satisfy the hypothesis of Proposition~\ref{prop:irred-tensor}, then
\[
  V(\psi_1)\widehat{\otimes} V(\psi_2) \cong V(\psi_1+\psi_2).
\]
In general, the following result follows by induction.

\begin{cor} \label{cor:multiple-h-tensor}
  Suppose that $V(\psi_1), \dotsc, V(\psi_k)$ are $\q\otimes A$-modules with pairwise disjoint supports. Then
  \[
    \widehat{\bigotimes}_{i=1}^{n}V(\psi_i)\cong V\left(\sum_{i=1}^{n}\psi_i\right).
  \]
\end{cor}

Now assume $\Gamma$ is a finite abelian group acting on both $\q$ and $A$ by automorphisms.  We also assume that $A$ is finitely generated and that $\Gamma$ acts freely on $\MaxSpec(A)$.

\begin{defn}[Evaluation map]
  Suppose $\sm_1,\dotsc, \sm_k$ are pairwise distinct maximal ideals of $A$. The associated \emph{evaluation map} is the composition
  \[
    \ev_{\sm_1,\dotsc, \sm_k} \colon \q\otimes A \twoheadrightarrow (\q\otimes A)/ \left(\q\otimes \prod_{i=1}^k \sm_i \right) \cong \bigoplus_{i=1}^k (\q\otimes A/\sm_i).
  \]
  We let $\ev_{\sm_1,\dotsc, \sm_k}^\Gamma$ denote the restriction of $\ev_{\sm_1,\dotsc, \sm_k}$ to $(\q\otimes A)^\Gamma$.
\end{defn}

  Let $\sm_1,\dotsc, \sm_k$ be pairwise distinct maximal ideals of $A$, and for each $i=1,\dotsc, k$, let $V_i$ be an irreducible finite-dimensional $\q\otimes A/\sm_i$-module, with corresponding representation $\rho_i \colon \q\otimes A/\sm_i \to \mathfrak{gl}(V_i)$. Then the representation given by the composition
  \begin{gather*}
    \q\otimes A \xrightarrow{\ev_{\sm_{1},\dotsc, \sm_k}}
    \bigoplus_{i=1}^k\left(\q\otimes (A/\sm_{i})\right) \xrightarrow{\widehat{\bigotimes}_{i=1}^k \rho_i}
    \End\left(\widehat{\bigotimes}_{i=1}^k V_i\right)
  \end{gather*}
  is denoted by
  \begin{equation}\label{irred.prod.ev.rep}
      \widehat{\ev}_{\sm_1,\dotsc, \sm_k}(\rho_1,\dotsc,\rho_k)
  \end{equation}
  and the corresponding module is denoted by
  \begin{equation}\label{irred.prod.ev.mod}
      \widehat{\ev}_{\sm_1,\dotsc, \sm_k}(V_1,\dotsc,V_k).
  \end{equation}
  We define $\widehat{\ev}^\Gamma_{\sm_1,\dotsc, \sm_k}(\rho_1,\dotsc,\rho_k)$ to be the restriction of $\widehat{\ev}_{\sm_1,\dotsc, \sm_k}(\rho_1,\dotsc,\rho_k)$ to $(\q\otimes A)^\Gamma$. The notation $\widehat{\ev}^\Gamma_{\sm_1,\dotsc, \sm_k}(V_1,\dotsc,V_k)$ is defined similarly.

If we consider tensor products instead of irreducible products, then the above are called \emph{evaluation representations} and \emph{evaluation modules}, respectively.

\begin{rmk} \label{rmk:single-point-decomp}
Observe that, by definition,
  \[
    \widehat{\ev}_{\sm_1,\dotsc,\sm_k}(\rho_1,\dotsc,\rho_k) \cong \widehat{\bigotimes}_{i=1}^k \ev_{\sm_i}(\rho_i).
  \]
\end{rmk}

\begin{prop} \label{prop:eval-reduced-support}
  An irreducible finite-dimensional representation of $\q\otimes A$ is isomorphic to a representation of the form~\eqref{irred.prod.ev.rep} if and only if it has finite reduced support.
\end{prop}

\begin{proof}
  Let $\rho$ be an irreducible finite-dimensional representation of $\q\otimes A$. Assume
  \[
  \rho\cong\widehat{\ev}_{\sm_1,\dotsc,\sm_k}(\rho_1,\dotsc,\rho_k),
  \]
where $\sm_1,\dotsc,\sm_k$ are pairwise distinct maximal ideals of $A$ and $\rho_i$ is an irreducible representation of $\q\otimes A/\sm_i$. Let $I=\prod_{i=1}^k \sm_i$.  Then $\Supp(I)=\{\sm_1,\dotsc, \sm_k \}$ and $\rho(\q\otimes I)=0$.  Thus $\rho$ has finite support. Furthermore we have that $\sqrt{I} = \bigcap_{i=1}^k \sm_i = \prod_{i=1}^k \sm_i = I$ and hence $I$ is a radical ideal.  This proves the forward implication.

  Suppose now that $\rho(\q\otimes I)=0$ for some radical ideal $I$ of $A$ of finite support. Thus $I=\sqrt{I}=\prod_{i=1}^{n}\sm_i$ for some distinct maximal ideals $\sm_1,\dotsc,\sm_k$ of $A$. Hence, $\rho$ factors through the map
  \[
    \q\otimes A \twoheadrightarrow (\q\otimes A)/ \left(\q\otimes \prod_{i=1}^k \sm_i\right)\cong \bigoplus_{i=1}^k (\q\otimes A/\sm_i).
  \]
  Then, by \cite[Prop.~8.4]{che95}, there exist irreducible finite-dimensional representations $\rho_i$ of $\q\otimes A/\sm_i$, $i=1,\dotsc,k$, such that
  \[
    \rho\cong\widehat{\bigotimes}_{i=1}^k \ev_{\sm_i}(\rho_i) \cong \widehat{\ev}_{\sm_1,\dotsc,\sm_k}(\rho_1,\dotsc,\rho_k).
  \]
  Thus $\rho$ is isomorphic to a representation of the form~\eqref{irred.prod.ev.rep}. This completes the proof of the reverse implication.
\end{proof}

\begin{defn}[$X_*$]
  Let $X_*$ denote the set of finite subsets $\sM\subseteq \MaxSpec(A)$ having the property that $\sm'\notin \Gamma \sm$ for distinct $\sm,\sm'\in \sM$.
\end{defn}

\begin{lemma}[{\cite[Lem.~5.6]{sav14}}] \label{lem:eval-surjective}
  If $\{\sm_1,\dotsc,\sm_k\} \in X_*$, then the map $\widehat{\ev}_{\sm_1,\dotsc,\sm_k}^\Gamma$ is surjective.
\end{lemma}

Let ${\mathcal R}(\q)$ denote the set of isomorphism classes of irreducible finite-dimensional representations of $\q$. Then $\Gamma$ acts on $\mathcal{R}(\q)$ by
\[
  \Gamma\times \mathcal{R}(\q)\rightarrow \mathcal{R}(\q),\quad (\gamma,[\rho])\mapsto \gamma [\rho]:=[\rho\circ \gamma^{-1}],
\]
where $[\rho]\in \mathcal{R}(\q)$ denotes the isomorphism class of a representation $\rho$ of $\q$.

\begin{defn}[$\mathcal{E}(\q, A)$, $\mathcal{E}(\q, A)^\Gamma$]
  Let $\mathcal{E}(\q, A)$ denote the set of finitely supported functions $\Psi \colon \MaxSpec(A) \rightarrow \mathcal{R}(\q)$ and let $\mathcal{E}(\q, A)^\Gamma$ denote the subset of $\mathcal{E}(\q, A)$ consisting of those functions that are $\Gamma$-equivariant.  Here the support of $\Psi$, denoted $\Supp(\Psi)$, is the set of all $\sm \in \MaxSpec(A)$ for which $\Psi(\sm) \neq 0$, where $0$ denotes the isomorphism class of the trivial (one-dimensional) representation.
\end{defn}

If $\rho$ and $\rho'$ are isomorphic representations of $\q$, then the representations $\ev_\sm(\rho)$ and $\ev_\sm(\rho')$ are also isomorphic, for any $\sm \in \MaxSpec A$.  Therefore, for $[\rho]\in \mathcal{R}(\q)$, we can define $\ev_\sm[\rho]$ to be the isomorphism class of $\ev_\sm(\rho)$, and this is independent of the representative $\rho$.  For $\Psi \in \mathcal{E}(\q, A)$ such that $\Supp(\Psi)=\{\sm_1,\dotsc, \sm_k\}$, we define $\widehat{\ev}_{\Psi}$ to be the isomorphism class of ${\widehat{\ev}_{\sm_1,\dotsc, \sm_k}}(\Psi(\sm_1),\dotsc,\Psi(\sm_k))$, which is well-defined by the above comments and Remark~\ref{rmk:single-point-decomp}. If $\Psi$ is the map that is identically $0$ on $\MaxSpec(A)$, then, by definition, $\widehat{\ev}_{\Psi}$ is the isomorphism class of the trivial (one-dimensional) representation of $\q\otimes A$.

\begin{lemma} \label{lem:evaluation-invariance}
  Let $\Psi\in {\mathcal E}(\q,A)^\Gamma$ and $\sm\in \MaxSpec(A)$. Then, for all $\gamma\in \Gamma$,
  \[
    \widehat{\ev}_\sm(\Psi(\sm)) = \widehat{\ev}_{\gamma \sm}(\gamma\Psi(\sm)) = \widehat{\ev}_{\gamma \sm}(\Psi({\gamma \sm})).
  \]
\end{lemma}

\begin{proof}
  The proof is the same of that in \cite[Lem.~4.13]{neh-sav-sen12}, where $\q$ is replaced by a finite-dimensional Lie algebra.
\end{proof}

\begin{defn}[${\widehat{\ev}}_{\Psi}^\Gamma$]
  Let $\Psi\in {\mathcal E}(\q,A)^\Gamma$ and let $\sM=\{\sm_1,\dotsc,\sm_k\}\in X_*$ contain one element from each $\Gamma$-orbit in $\Supp(\Psi)$.  We define ${\widehat{\ev}}_{\Psi}^\Gamma:={\widehat{\ev}}_{\sm_1,\dotsc,\sm_k}^\Gamma (\Psi(\sm_1),\dotsc,\Psi(\sm_k))$. By Lemma~\ref{lem:evaluation-invariance}, ${\widehat{\ev}}_{\Psi}^\Gamma$ is independent of the choice of $\sM$. If $\Psi=0$, we define ${\widehat{\ev}}_{\Psi}^\Gamma$ to be the isomorphism class of the trivial (one-dimensional) representation of $(\q\otimes A)^\Gamma$.
\end{defn}

\begin{prop}\label{prop:correspondence-injective}
  The map $\Psi\mapsto \widehat{\ev}_{\Psi}$ from $\mathcal{E}(\q, A)$ to the set of isomorphism classes of irreducible finite-dimensional representations of $\q\otimes A$ is injective.
\end{prop}

\begin{proof}
  If $\Psi\neq \Psi'\in\mathcal{E}(\q, A)$, then there exists $\sm\in \MaxSpec (A)$ such that $\Psi(\sm)\neq \Psi'(\sm)$.  Without loss of generality, we may assume that $\Psi(\sm) \neq 0$. Let $\Supp(\Psi)\cup \Supp(\Psi')=\{\sm_1,\dotsc,\sm_{k}\}$, where $\sm=\sm_1$ and consider the following ideal of $A$:
  \[
    I=\sm_2 \dotsm\sm_k.
  \]
  Note that $\mathfrak{a}=\q\otimes I$ is a Lie subalgebra of $\q \otimes A$ such that
  $\ev_{\sm}(\mathfrak{a})\cong\q$ and $\ev_{\sm_j}(\mathfrak{a})=0$ for $j=2,\dotsc,k$.

  Suppose that $\widehat{\ev}_\Psi\cong \widehat{\ev}_{\Psi'}$, and define
  \[
    \rho :=\ev_{\sm_2}(\Psi(\sm_2))\widehat{\otimes}\cdots\widehat{\otimes}\ev_{\sm_k}(\Psi(\sm_k)) \quad \text{and} \quad \rho':=\ev_{\sm_2}(\Psi'(\sm_2))\widehat{\otimes}\cdots\widehat{\otimes}\ev_{\sm_k}(\Psi'(\sm_k)),
  \]
  with associated modules $V$ and $V'$, respectively.  Then $\rho(\mathfrak{a})=\rho'(\mathfrak{a})=0$.  We divide the proof into three cases.

  For the first case, assume that we have isomorphisms of $\q \otimes A$-modules
  \[
    \ev_{\sm_1}(\Psi(\sm_1))\otimes \rho \cong \hat{\rho} \oplus \hat{\rho} \quad\text{and}\quad \ev_{\sm_1}(\Psi'(\sm_1))\otimes \rho' \cong \hat{\rho}' \oplus \hat{\rho}',
  \]
  where $\hat{\rho}$ and $\hat{\rho}'$ are subrepresentations of $\ev_{\sm_1}(\Psi(\sm_1))\otimes \rho$ and $\ev_{\sm_1}(\Psi'(\sm_1))\otimes \rho'$, respectively.  Since $\widehat{\ev}_\Psi\cong \widehat{\ev}_{\Psi'}$, we must have $\hat{\rho} \cong \hat{\rho}'$, and so
  \begin{multline*}
    \ev_{\sm_1}(\Psi(\sm_1))^{\oplus \dim V} \cong (\ev_{\sm_1}(\Psi(\sm_1)) \otimes \rho)|_\mathfrak{a} \cong (\hat \rho \oplus \hat \rho)|_\mathfrak{a}
    \cong (\hat \rho' \oplus \hat \rho')|_\mathfrak{a} \\
    \cong (\ev_{\sm_1}(\Psi'(\sm_1)) \otimes \rho')|_\mathfrak{a}  \cong \ev_{\sm_1}(\Psi'(\sm_1))^{\oplus \dim V'},
  \end{multline*}
  where the first isomorphism follows from the fact that $\rho(\mathfrak{a})=0$ and the last follows from the fact that $\rho'(\mathfrak{a})=0$.  But this is a contradiction, since $\Psi(\sm_1)\neq \Psi'(\sm_1)$.

  For the second case, assume
  \[
    \ev_{\sm_1}(\Psi'(\sm_1))\otimes \rho' \text{ is irreducible } \quad\text{and}\quad \ev_{\sm_1}(\Psi(\sm_1))\otimes \rho \cong \hat \rho \oplus \hat \rho,
  \]
  where $\hat \rho \subseteq \ev_{\sm_1}(\Psi(\sm_1)) \otimes \rho$ is a subrepresentation. Thus $\hat \rho \cong \ev_{\sm_1}(\Psi'(\sm_1)) \otimes \rho'$, which implies that
  \[
    \ev_{\sm_1}(\Psi(\sm_1))^{\oplus \dim V} \cong (\hat \rho \oplus \hat \rho)|_\mathfrak{a}
    \cong (\ev_{\sm_1}(\Psi'(\sm_1))\otimes \rho')^{\oplus 2}|_\mathfrak{a} \cong \ev_{\sm_1}(\Psi'(\sm_1))^{\oplus 2\dim V'}.
  \]
  So again we have a contradiction.

  The remaining case, when both $\ev_{\sm_1}(\Psi'(\sm_1))\otimes \rho'$ and $\ev_{\sm_1}(\Psi(\sm_1))\otimes \rho$ are irreducible $\q \otimes A$-modules, is similar.
\end{proof}

\begin{cor} \label{cor:twisted-eval-injective}
  For all $\Psi \in \mathcal{E}(\q,A)^\Gamma$, we have that ${\widehat{\ev}}_\Psi^\Gamma$ is the isomorphism class of an irreducible finite-dimensional representation.  Furthermore, the map $\Psi\mapsto \widehat{\ev}_{\Psi}^\Gamma$ from $\mathcal{E}(\q, A)^\Gamma$ to the set of isomorphism classes of irreducible finite-dimensional representations of $(\q\otimes A)^\Gamma$ is injective.
\end{cor}

\begin{proof}
  The first statement follows from Lemma~\ref{lem:eval-surjective} and the definition of the irreducible product.  Suppose $\Psi,\Psi' \in \mathcal{E}(\q, A)^\Gamma$ such that $\widehat{\ev}_\Psi^\Gamma = \widehat{\ev}_{\Psi'}^\Gamma$.  Let $\sM =\{\sm_1,\dotsc,\sm_k\} \in X_*$ contain one element of each $\Gamma$-orbit in $\Supp (\Psi) \cup \Supp (\Psi')$.  Then $\widehat{\ev}_\Psi^\Gamma$ and $\widehat{\ev}_{\Psi'}^\Gamma$ are the restrictions to $(\g \otimes A)^\Gamma$ of $\widehat{\ev}_{\sm_1,\dotsc,\sm_k}(\Psi(\sm_1),\dotsc,\Psi(\sm_k))$ and $\widehat{\ev}_{\sm_1,\dotsc,\sm_k}(\Psi'(\sm_1),\dotsc,\Psi'(\sm_k))$, respectively.  By Lemma~\ref{lem:eval-surjective}, it follows that $\widehat{\ev}_{\sm_1,\dotsc,\sm_k}(\Psi(\sm_1),\dotsc,\Psi(\sm_k)) = \widehat{\ev}_{\sm_1,\dotsc,\sm_k}(\Psi'(\sm_1),\dotsc,\Psi'(\sm_k))$.  Then, by Proposition~\ref{prop:correspondence-injective}, we have $\Psi(\sm_i) = \Psi'(\sm_i)$ for $i=1,\dotsc,k$.  Thus $\Psi = \Psi'$.
\end{proof}

\begin{rmk}
  If the target Lie superalgebra $\q$ is replaced by a Lie algebra or a basic classical Lie superalgebra $\g$, then the tensor product of irreducible finite-dimensional representations with disjoint supports is always irreducible (see \cite[Prop.~4.9]{neh-sav-sen12} for Lie algebras and \cite[Prop.~4.12]{sav14} for basic classical Lie superalgebras). In particular, the evaluation representation $\ev_\Psi$ is an irreducible finite-dimensional representation for all $\Psi\in {\mathcal E}(\g,A)$, where $\ev_\Psi$ is defined by replacing the irreducible product by the tensor product in the definition of $\widehat{\ev}_\Psi$.
\end{rmk}

\section{Classification of finite-dimensional representations}
\label{sec:fd-rep-classification}

In this section we present our main result: the classification of the irreducible finite-dimensional $\q\otimes A$-modules and $(\q \otimes A)^\Gamma$-modules.  We assume that $A$ is finitely generated.

\begin{thm} \label{thm:untwisted-classification}
  The map
  \begin{equation} \label{eq:classification-bijection}
    \mathcal{E}(\q, A)  \rightarrow \mathcal{R}(\q\otimes A),\qquad  \Psi \mapsto \widehat{\ev}_\Psi,
  \end{equation}
  is a bijection, where $\mathcal{R}(\q\otimes A)$ is the set of isomorphism classes of irreducible finite-dimensional representations of $\q\otimes A$. In particular, all irreducible finite-dimensional representations are representations of the form \eqref{irred.prod.ev.rep}.
\end{thm}

\begin{proof}
  By Proposition~\ref{prop:correspondence-injective}, it is enough to show that all irreducible finite-dimensional representations of $\q \otimes A$ are of the form \eqref{irred.prod.ev.rep}.  Thus, it suffices, by Proposition~\ref{prop:eval-reduced-support}, to show that, for every irreducible finite-dimensional $\q \otimes A$-module $V$, we have $(\q \otimes J)V=0$ for some radical ideal $J\subseteq A$ of finite-codimension.

  By Corollary~\ref{cor:fd-finite-codim-ideal}, we have that $(\q\otimes I)V=0$ for some ideal $I$ of $A$ of finite codimension.  Let $J=\sqrt{I}$ be the radical of $I$.  To prove that $(\q\otimes J)V=0$, it suffices, by Lemma~\ref{lem:ideal-annihilate-vector}, to show that $(\q\otimes J)v=0$ for some nonzero vector $v \in V$.

  Consider now $V$ as a $\q\otimes A/I$-module.  We will show that $(\q\otimes (J/I))v=0$ for some nonzero $v\in V$. Since $A$ is finitely generated, and hence Noetherian, we have $J^k \subseteq I$ for some $k \in \N$, by Lemma~\ref{lem:assoc-alg-facts}\eqref{lem-item:Noetherian-power-radical}.  Hence, $(\q\otimes (J/I))^{(k)}=\q^{(k)}\otimes (J^k/I)=0$, and so $\q\otimes (J/I)$ is solvable. On the other hand, since $\q_{\bar{0}}$ is a simple Lie algebra, we have
  \begin{multline*}
    [(\q\otimes (J/I))_{\bar{1}},(\q\otimes (J/I))_{\bar{1}}] = [\q_{\bar 1},\q_{\bar 1}]\otimes (J^{2}/I)
     \subseteq \q_{\bar 0}\otimes (J^{2}/I) \\
     = [\q_{\bar 0},\q_{\bar 0}]\otimes (J^{2}/I)
     = [(\q\otimes (J/I))_{\bar{0}},(\q\otimes (J/I))_{\bar{0}}].
  \end{multline*}
  Then, by Lemma~\ref{lem:solvablemod}, there exists a one-dimensional $\q\otimes (J/I)$-submodule of $V$. Thus, we have a nonzero vector $v\in V$ and $\theta \in (\q\otimes J)^*$, such that
  \[
    \mu v=\theta(\mu)v,\quad \text{for all } \mu \in \q\otimes J.
  \]
  We want to prove that $\theta=0$. If $\mu\in \mathfrak{n}^{\pm}\otimes J$, then $\theta(\mu)^mv=\mu^mv=0$ for $m$ sufficiently large, since $V$ is finite dimensional and hence has a finite number of nonzero weight spaces. Thus $\theta(\mathfrak{n}^\pm\otimes J)=0$.  It remains to show that $\theta(\h\otimes J)=0$. Denote by $\theta'$ the restriction of $\theta$ to $\q_{\bar{0}}\otimes J$. Then $\theta'$ defines a one-dimensional representation of the Lie algebra $\q_{\bar{0}}\otimes J$, and hence the kernel of $\theta'$ must be an ideal of $\q_{\bar{0}}\otimes J$ of codimension at most one. Because $\q_{\bar{0}}$ is a simple finite-dimensional Lie algebra, it is easy to see that this kernel must be all of $\q_{\bar{0}}\otimes J$, and hence $\theta'=0$. Since $\h_{\bar{0}}\subseteq \q_{\bar{0}}$, we also have that $\theta(\h_{\bar{0}}\otimes J)=0$. Therefore, Lemma~\ref{lem:even-odd-ideal-kill} implies that $(\h\otimes J)v=0$.
\end{proof}

Now assume $\Gamma$ is a finite abelian group acting on both $\q$ and $A$ by automorphisms.  We also assume that $\Gamma$ acts freely on $\MaxSpec(A)$.

\begin{prop} \label{prop:twisted-are-restrictions}
  Every finite-dimensional $(\q\otimes A)^\Gamma$-module $V$ is the restriction of a $\q\otimes A$-module $V'$ whose support is an element of $X_*$.  Furthermore, $V$ is irreducible if and only if $V'$ is.
\end{prop}

\begin{proof}
  The proof is the same as the proof of \cite[Prop.~8.5]{sav14}.  Although that reference assumes that the target Lie superalgebra $\g$ is basic classical, the proof of this result only requires $\g$ to be a simple finite-dimensional Lie superalgebra.  Note also that the statement of \cite[Prop.~8.5]{sav14} does not include the fact that the support of $V'$ is an element of $X_*$.  However, this property is demonstrated in the proof.
\end{proof}

\begin{thm} \label{thm:twisted-classification}
  Suppose $A$ is a finitely generated unital associative $\C$-algebra and $\Gamma$ is a finite abelian group acting on $A$ and $\q$ by automorphisms.  Furthermore, suppose that the induced action of $\Gamma$ on $\MaxSpec (A)$ is free.  Then the map
  \begin{equation} \label{eq:twisted-classification-map}
    \mathcal{E}(\q,A)^\Gamma \rightarrow {\mathcal R}(\q,A)^\Gamma,\quad \Psi \mapsto \widehat\ev_\Psi^\Gamma,
  \end{equation}
  is a bijection, where ${\mathcal R}(\q,A)^\Gamma$ is the set of isomorphism classes of irreducible finite-dimensional representations of $(\q\otimes A)^\Gamma$.
\end{thm}

\begin{proof}
  The map \eqref{eq:twisted-classification-map} is surjective by Proposition~\ref{prop:twisted-are-restrictions}, while injectivity follows from Corollary~\ref{cor:twisted-eval-injective}.
\end{proof}



\newcommand{\doi}[1]{DOI: \href{http://dx.doi.org/#1}{\tt \nolinkurl{#1}}}

\end{document}